\documentclass[12pt]{amsart}
\usepackage{fullpage}
\usepackage{float}
\usepackage{amsmath,amssymb}
\usepackage{amsthm, amsmath, amsfonts, amssymb, verbatim, graphicx, colonequals, mathrsfs, tikz-cd}
\usepackage{booktabs}
\usepackage{xcolor}
\usepackage{array}

\usepackage{hyperref}

\usepackage{float}

\usepackage{url}
\usepackage[alphabetic]{amsrefs}
\usepackage{wrapfig}

\theoremstyle{plain}
\newtheorem{thm}{Theorem}
\newtheorem{lem}[thm]{Lemma}

\newtheorem{cor}[thm]{Corollary}
\theoremstyle{definition}
\newtheorem{defe}[thm]{Definition}

\theoremstyle{remark}
\newtheorem{rem}[thm]{Remark}

\newcommand{\nc}{\newcommand}

\renewcommand{\leq}{\leqslant}

\nc{\ssec}{\subsection}
\nc {\fX}{\mathfrak{X}}
\nc {\tY}{\mathtt{Y}}
\nc {\fU}{\mathfrak{U}}
\nc {\fR}{\mathfrak{R}}

\nc {\cP}{\mathcal{P}}
\nc {\cI}{\mathcal{I}}
\nc {\cT}{\mathcal{T}}

\nc {\bQ}{\mathbb{Q}}

\nc {\Irr} {\mathrm{Irr}}
\nc {\Aut} {\mathrm{Aut}}
\nc {\GL} {\mathrm{GL}}
\nc {\PGL} {\mathrm{PGL}}
\nc {\rank} {\mathrm{rank}}
\nc {\Hom}{\mathrm{Hom}}
\nc {\Spec}{\mathrm{Spec} \,}

\nc{\bbG}{{\widehat{\Psi}}}
\nc{\bbP}{\mathbb{P}}
\nc{\bbH}{\mathbb{H}}
\nc{\bbT}{\mathbb{T}}
\nc{\bbC}{\mathbb{C}}
\nc{\oFp}{\overline{\mathbb{F}_p}}

\nc{\bC}{\mathbb{C}}
\nc{\bG}{\mathbf{G}}
\nc{\bY}{\mathbf{Y}}
\nc{\bZ}{\mathbf{Z}}
\nc{\bB}{\mathbf{B}}
\nc{\bT}{\mathbf{T}}
\nc{\bS}{\mathbf{S}}
\nc{\bH}{\mathbf{H}}
\nc{\bK}{\mathbf{K}}
\nc{\bX}{\mathbf{X}}
\nc{\bR}{\mathbf{R}}
\nc{\bA}{\mathbf{A}}

\nc {\cM}{\mathcal{M}}
\nc {\bF}{\mathbb{F}}
\nc{\Fq}{{\mathbb{F}_q}}
\nc {\bbZ}{\mathbb{Z}}
\nc {\bbR}{\mathbb{R}}
\nc {\bbQ}{\mathbb{Q}}
\nc {\bbX}{\mathbb{X}}

\nc {\tG}{\mathscr{G}}
\nc {\tB}{\mathscr{B}}
\nc {\tZ}{\mathscr{Z}}
\nc {\tT}{\mathscr{T}}
\nc {\tX}{\mathscr{X}}
\nc {\tH}{\mathtt{H}}

\nc {\ra}{\rightarrow}

\begin{document} 
\title{Character stacks are PORC count} 
 
\author{Nick Bridger} 
\author{Masoud Kamgarpour}


\address{School of Mathematics and Physics, The University of Queensland} 
\email{nicholas.bridger@uq.net.au}
\email{masoud@uq.edu.au}

\date{\today}

\begin{abstract} 
We compute the number of points over finite fields of the character stack associated to a compact surface group and a reductive group with connected centre. We find that the answer is a Polynomial On Residue Classes (PORC). The key ingredients in the proof are Lusztig's Jordan decomposition of complex characters of finite reductive groups and Deriziotis's results on their genus numbers . As a consequence of our main theorem, we obtain an expression for the $E$-polynomial of the character stack. 
\end{abstract} 

\maketitle 

\tableofcontents

\section{Introduction} 
Let $\Gamma$ be the fundamental group of a Riemann surface and $G$ a reductive group.  
The \emph{character stack} associated to $(\Gamma, G)$ is  the quotient stack 
\begin{equation} \label{eq:charStack}
\fX \colonequals [\Hom(\Gamma,G)/G].
\end{equation} 
 This space and its cousins (the character variety, moduli of stable Higgs bundles and moduli of flat connections) play a central role in diverse areas of mathematics such as non-abelian Hodge theory \cites{SimpsonICM, SimpsonModuli} and the geometric Langlands program \cites{BD, BenZviNadler}. 
 
 The study of the topology and geometry of these spaces has been a subject of active research for decades. 
In their ground breaking work \cite{HRV}, Hausel and Villegas counted points on the character stack associated to the once-punctured surface group and $G=\GL_n$, where the loop around the punctured is mapped to a primitive root of unity. This gave rise to much further progress in understanding the arithmetic geometry of character stacks, cf. 
\cites{HLRV, dCHM, Letellier, Mereb, BaragliaHekmati, Mellit, LetellierRodriguez, Ballandras}. 

Almost all the previous work in this area concerns the case when $G=\GL_n$, $\mathrm{SL}_n$, or  $\PGL_n$. The only exception we know of is the unpublished thesis \cite{Cambo}. Note that from the point of view of Langlands correspondence, it is crucial to understand character stacks of all reductive groups, for Langlands central conjecture, functoriality, concerns relationship between automorphic functions (or sheaves) of different reductive groups. 

The purpose of this paper is to study the arithmetic geometry of the character stack associated to a compact surface group and an arbitrary reductive group $G$ with connected centre. This represents the first step in generalising the program of Hausel--Letellier--Villegas \cites{HRV, HLRV, Letellier, LetellierRodriguez}  from type $A$ to arbitrary type in a uniform manner.

\subsection{Main result} To state our main result, we need a definition regarding counting problems whose solutions are Polynomial On Residue Classes (PORC), cf. \cite{Higman}.

\subsubsection{}  
Let $Y$ be a map from finite fields to finite groupoids.  For instance, $Y$ can be a scheme or a stack of finite type over $\bbZ$. We write $|Y(\Fq)|$ for the groupoid cardinality of $Y(\Fq)$; i.e., 
\[
|Y(\Fq)| \colonequals \sum_{y\in Y(\Fq)} \frac{1}{|\mathrm{Aut}(y)|}.
\]

\begin{defe} \label{d:PORC}
We say $Y$ is \emph{PORC count} if there exists an integer $d$, called the \emph{modulus},  and polynomials $|\!|Y|\!|_0, \cdots, |\!|Y|\!|_{d-1}\in \mathbb{C}[t]$ such that 
\[
|Y(\Fq)| = |\!|Y|\!|_i(q),\qquad \forall \,\,q\equiv i \mod d. 
\]
\end{defe} 

For instance, $\mathrm{Spec}\,  \mathbb{Z}[x]/(x^2+1) $ is PORC count with modulus $4$ 
and counting polynomials $|\!|X|\!|_{0}=|\!|X|\!|_{2}=1$, $|\!|X|\!|_{1}=2$, and $|\!|X|\!|_{3}=0$.

\subsubsection{} 
Now let $\Gamma_g$ be the fundamental group of a compact Riemann surface of genus $g\geq 1$, $G$ a connected (split) reductive group over $\bbZ$,  $G^\vee$ the (Langlands) dual group, and
$\fX$ the character stack associated to $(\Gamma_g, G)$ as in \eqref{eq:charStack}.   

 \begin{thm} \label{t:main} If $G$ has connected centre, then  $\fX$ is PORC count with the modulus $d(G^\vee)$ and counting polynomials $|\!|\fX|\!|_0,\cdots |\!|\fX|\!|_{d(G^\vee)-1}$ defined in, respectively, Definitions \ref{d:modulus} and \ref{d:main}\footnote{The expression for the counting polynomials is explicit in so far as the genus numbers are explicit; see \S \ref{s:genus}.}.
 \end{thm} 

This theorem refines \cite{LS}*{Theorem 1.2} which gave an asymptotic for $|\fX(\Fq)|$. 

\subsubsection{} \label{sss:Frobenius} Let us outline the proof of the theorem. First, using Lang's theorem, 
it is easy (cf. \cite{Behrend}*{2.5.1})  to show that 
\begin{equation} 
|\fX(\Fq)|=|\Hom(\Gamma_g, G(\Fq))/G(\Fq)|.
\end{equation}
Thus, our goal is to count the number of homomorphisms $\Gamma_g\ra G(\Fq)$; i.e.,  the number of solutions to the equation 
$[x_1,y_1]\cdots [x_g, y_g]=1$ 
in the finite group $G(\Fq)$. 

\subsubsection{} 
Next, a theorem going back to Frobenius (cf. \cite{HRV}*{\S 2.3}) states that  
\begin{equation} \label{eq:Frobenius} 
|\Hom(\Gamma_g, G(\Fq))/G(\Fq)| = \sum_{\chi \in \Irr(G(\Fq))} \bigg(\frac{|G(\Fq)|}{\chi(1)}\bigg)^{2g-2}. 
\end{equation} 
Here $\Irr(G(\Fq))$ denote the set of irreducible complex characters of $G(\Fq)$. 
Thus, computing $|\fX(\Fq)|$ is a problem in complex representation theory of finite reductive groups.

\subsubsection{} According to Lusztig's Jordan decomposition \cite{Lusztig84}, there is a bijection between $\Irr(G(\Fq))$ and the set of pairs $([s], \rho)$ consisting of conjugacy classes $[s]$ of semisimple elements in the dual group $G^\vee(\Fq)$ and irreducible unipotent characters $\rho$ of the centraliser $G^\vee_s(\Fq)$.  The parameterisation and degrees of unipotent representations were also determined by Lusztig \cite{Lusztig84}. Hence, it remains to understand centralisers of semisimple elements of $G^\vee(\Fq)$.

\subsubsection{} The final  ingredient in the proof is results of Carter and Deriziotis on centralisers of semisimple elements and genus numbers \cites{Carter1978, Deriziotis85}. The notion of ``genus'' is a reductive generalisation of Green's notion of ``type". The latter is ubiquitous in point counts on character varieties in type $A$ \cites{HRV, HLRV, Letellier, Mereb, LetellierRodriguez}. The term ``genus number" refers to the number of conjugacy classes of semisimple elements whose centraliser is in the same conjugacy class. 
The fact \cite{Deriziotis85} that genus numbers of reductive groups (with connected centre) are PORC is the reason that character stacks are PORC count.

\subsubsection{Remarks} 
\begin{enumerate} 
\item[(i)] 
It is well-known that for every $\chi\in \Irr(G(\Fq))$, the quotient $|G(\Fq)|/\chi(1)$ is a polynomial in $q$, cf. \cite{GeckMalle}*{Rem. 2.3.27}. Theorem \ref{t:main} is, however, not trivial because the sum in \eqref{eq:Frobenius} is over a set which  depends on $q$. 
\item[(ii)] Consider the representation $\zeta$-function of $G(\Fq)$ defined by 
\[
\zeta_{G(\Fq)}(s) \colonequals \sum_{\chi\in \Irr(G(\Fq))} \chi(1)^{-s}. 
\]
Then Frobenius' theorem \eqref{eq:Frobenius} can be reformulated as 
\[
|\Hom(\Gamma_g, G(\Fq))/G(\Fq)|= \zeta_{G(\Fq)}(2g-2) |G(\Fq)|^{2g-2}.
\] 
Our approach gives an explicit expression for $\zeta_{G(\Fq)}(s)$ for any $s$; see \ref{sss:zeta}. 
\item[(iii)] Note that $d(\GL_n)=1$ thus the $\GL_n$-character stack is polynomial count (see below). 
\end{enumerate}

\subsection{Consequences} We now discuss some of the corollaries of our main theorem. 
Recall the following:
\begin{defe} \label{d:poly}
An algebraic stack $Y$ of finite type over $\Fq$ is called \emph{polynomial count}\footnote{For stacks, it is more natural to consider \emph{rational count} objects \cite{LetellierRodriguez}, but it turns out that all the stacks we consider are actually polynomial count, so we restrict to this case.}  if there exists a polynomial $|\!|Y|\!|$ such that 
\[
|Y(\mathbb{F}_{q^n})|=|\!|Y|\!|(q^n),\qquad \forall n \in \mathbb{N}.  
\]
\end{defe} 
By \cite{LetellierRodriguez}*{Theorem 2.8}, if a quotient stack $Y=[R/G]$, with $G$ connected, is polynomial count, then the $E$-series of $Y$ is a well-defined polynomial  and it equals $|\!|Y|\!|$. In particular, one finds that the dimension, the number of irreducible components of maximal dimension, and the Euler characteristic\footnote{By the Euler characteristic of the stack $Y$, we mean the alternating sum of dimensions of compactly supported cohomology groups $H_c^i(Y_{\overline{\Fq}}; \overline{\mathbb{Q}_\ell})$, when this sum makes sense.} of $Y$ equal, respectively, the degree, the leading coefficient, and the value at $1$ of the polynomial $|\!|Y|\!|$.

\subsubsection{} Now let $\fX_\Fq \colonequals \fX\otimes_\bbZ \Fq$. As an immediate corollary of our main theorem, we obtain: 

\begin{cor} Suppose $q\equiv 1 \mod d(G^\vee)$. Then $\fX_\Fq$ is  polynomial count with counting polynomial $|\!|\fX|\!|_1$.\footnote{If $q$ is co-prime to $d(G^\vee)$,  then $\fX_\Fq$ becomes polynomial count after a finite base change.} Thus, the $E$-series of $\fX$ equals $|\!|\fX|\!|_1$. 
\end{cor}

\subsubsection{}  Let $\rank(G)$ denote the reductive rank of $G$. 
Analysing the leading term of the polynomial $|\!|\fX|\!|_1$, we find:  
\begin{cor}\label{c:main} Suppose $q\equiv 1 \mod d(G^\vee)$. Then 
\begin{enumerate} 
\item[(i)] If $g=1$ then $\dim(\fX_\Fq) = \rank  (G)$ and $\fX_\Fq$ has a unique irreducible component of maximal dimension. \label{itemi}
\item[(ii)] If $g>1$ then $\dim(\fX_\Fq) = (2g-2)\dim(G)+\dim(Z(G^\vee))$ and $\fX_\Fq$ has $|\pi_1([G,G])|$ irreducible components of maximal dimension. 
\end{enumerate} 
\end{cor} 

As observed in \cite{LS}*{Corollary 1.11}, the result holds without any assumption on $q$ because the Lang--Weil estimate implies that only the asymptotics of $|\fX(\Fq)|$ matters. Over the complex numbers, the above numerical invariants have also been understood from other perspectives; see the Appendix for further discussions.

\subsubsection{} The Euler characteristic of $\fX_\Fq$ is more subtle and has not been considered in the literature. In this direction, we have: 

\begin{cor}\label{c:Euler} 
Suppose $q\equiv 1 \mod d(G^\vee)$. Then 
\begin{enumerate} 
\item[(i)] If $g=1$ and $G$ is a simple adjoint group of type $G_2, F_4, E_6, E_7, E_8$, then $\chi(\fX_\Fq)$ equals $12, 56, 46, 237, 252$, respectively. 
\item[(ii)] If $g>1$ and $G$ is simple adjoint group of type $B_2$ or $G_2$, then $\chi(\fX_\Fq)$ equals $2^{8g-7}$ and $72^{2g-2} + 8^{2g-2} + 2\times 9^{2g-2}$, respectively. 
\item[(iii)] If $g>1$ then the Euler characteristic of the component of the $\PGL_n$-character stack associated to $1$ is equal to $\varphi(n)n^{2g-3}$, where $\varphi$ is the Euler  totient function. 
\end{enumerate} 
\end{cor}

Note that when $g=1$, $|\fX(\Fq)|$ equals the number of conjugacy classes of $G(\Fq)$. Thus, Part (i) can be verified by consulting tables of conjugacy classes of $G(\Fq)$. Part (ii) can also be verified by consulting character tables of $G(\Fq)$ for groups of small rank (cf. \cite{L}) but we found it instructive to prove this using our approach; see \S \ref{s:Examples}. 
Part (iii) should be compared with \cite{HRV}*{Corollary 1.1.1} which states that the Euler characteristic of a component of $\PGL_n$-character stack labelled by a \emph{primitive} root of unity is $\mu(n)n^{2g-3}$.

\subsection{Further directions} We expect the main theorem to hold for general reductive groups. The main difficulty with reductive groups with disconnected centre is that Lusztig's Jordan decomposition and genus numbers are more complicated because centralisers of semisimple elements in $G^\vee$ may be disconnected.

We also expect the theorem to hold for fundamental groups of non-orientable surfaces. For $G=\GL_n$, this is proved in \cite{LetellierRodriguez}. The main issue for general types is that the relationship between Frobenius--Schur indicators and the Lusztig--Jordan decomposition is not well-understood, cf. \cite{Vinroot} for some results in this direction.

Finally, we expect the theorem to hold for fundamental groups of punctured Riemann surfaces. In this case, a careful choice of conjugacy classes at the punctures (generalising the notion of generic from \cite{HLRV}) will ensure that the resulting character stack and character variety are the same. This is the subject of work in progress \cite{NAM}.

 \subsection{Structure of the text} 
In \S \ref{s:roots}, we review standard concepts regarding root datum and reductive groups over finite fields. In \S \ref{s:genus}, we recall some results of Carter and Deriziotis on centralisers of semsimple elements and genus numbers. In \S \ref{s:Lusztig}, we review Lusztig's Jordan decomposition of irreducible characters and classification of unipotent representations. Theorem \ref{t:main} and Corollary \ref{c:main} are proved in \S \ref{s:main}. In \S \ref{s:Examples}, we provide explicit formulas for the counting polynomials of character stacks associated to simple groups of semisimple rank $\leq 2$. In \S \ref{s:GLn}, we use Green's classification of irreducible characters of $\GL_n(\Fq)$ to count points on the character stack associated to $\GL_n$. Finally, in the appendix, we discuss the implications of our results for character stacks over $\bbC$.

\subsection{Acknowledgements} 
The second author would like to thank David Baraglia for introducing him to the world of character varieties and answering numerous questions and for Jack Hall for answering questions about stacks. We thank George Lusztig whose crucial comment set us on the right path. We also thank the participants of the Australian National University workshop ``Character varieties, $E$-polynomials and Representation  $\zeta$-functions", where our results were presented.

NB is supported by an Australian Government RTP Scholarship. MK is supported by Australian Research Council Discovery Projects.  The contents of this paper form a part of NB's Master's Thesis.

\section{Reductive groups over finite fields} \label{s:roots} 
In this section, we recall some basic notation and facts about structure of reductive groups over finite fields, cf. \cites{CarterBook, DigneMichel, GeckMalle}. But first, we define the notion of the modulus of a root datum used in our main theorem.

 \subsection{Modulus} \label{ss:modulus} 
Let  $\Psi=(X, X^\vee, \Phi, \Phi^\vee)$ be a root datum. Here, $X$ denotes the characters, $X^\vee$ cocharacters, $\Phi$ roots, and $\Phi^\vee$ coroots. 

  \begin{defe}  \label{d:modulus} 
 We define the \emph{modulus of $\Psi$}, denoted by $d(\Psi)$, to be the lcm of the sizes of torsion parts of the abelian groups $X/\langle \Phi_1\rangle $, where $\Phi_1$ ranges over closed subsystems of $\Phi$ and $\langle \Phi_1 \rangle$ denotes the subgroup of $X$ generated by $\Phi_1$. 
 \end{defe}

\subsubsection{} Let $G$ be a connected split reductive group with root datum $\Psi$. Then we define the modulus of $G$ by  $d(G) \colonequals d(\Psi)$. Note that the root datum $(X, X^\vee, \Phi_1, \Phi_1^\vee)$ defines a connected reductive subgroup $G_1\subseteq G$ of maximal rank.  The size of the torsion part of $X/\langle \Phi_1\rangle$ equals the number of connected components of the centre $Z(G_1)$. Thus, 
\begin{equation} 
d(G)=\mathrm{lcm} |\pi_0(Z(G_1))|, 
\end{equation}
 where $G_1$ ranges over connected reductive subgroups of $G$ of maximal rank. In particular, we see that $d(\GL_n)=1$. 
 
 \subsubsection{} Let $G$ be a simple simply connected group. Then one can show that $d(G)$ equals the lcm of coefficients of the highest root and the order of $Z(G)$, cf. \cite{Deriziotis85}. Thus, we have:
  \[
  \begin{array}{|c|c|c|c|c|c|c|c|c|c|}\hline
\mathrm{Type} & A_n  & B_n  & C_n  & D_n  & E_6 & E_7 & E_8 & F_4 & G_2\\\hline 
d(G)  & n+1 & 2 & 2  & 4 & 6  &  12 & 60 & 12  & 6 \\\hline 
 \end{array}
\]
Note that for types $B_n$, $C_n$, $E_6$, $G_2$ (resp. $D_n$, $E_7$, $E_8$),  $d(G)$ is the product of bad primes (resp. twice the product of bad primes) of $G$. We refer the reader to \cite{SpringerSteinberg} for the definition of bad primes.

 \subsection{Reductive groups over finite fields}\label{ss:relative} 
 Let $p$ be a prime, $k$ an algebraic closure of $\mathbb{F}_p$, and $\Fq$ the subfield of $k$ with $q$ elements. 
We use bold letters such as $\bX$ for schemes, stacks, etc. over $\Fq$ and script letters such as $\tX$ for their base change to $k$.  The (geometric) Frobenius $F=F_\tX \colon  \tX \ra \tX$ is the map $F_0\otimes \mathrm{id}$, where $F_0$ is the endomorphism of $\bX$ defined by raising the functions on $\bX$ to the $q^\mathrm{th}$ power.

\subsubsection{} 
Let $\bG$ 
 be a connected reductive group over $\Fq$ with a maximal quasi-split torus $\bT$.  Let $\Psi=(X, X^\vee, \Phi, \Phi^\vee)$ denote the root datum of $(\tG, \tT)$. We now explain how to encode the rational structure of $\bG$  via the root datum. 
The Frobenius  $F \colon \tT \ra \tT$ induces a homomorphism on characters
\[
X \to X, \qquad  \lambda \mapsto \lambda \circ F, 
\] 
which we denote by the same letter. We then have an automorphism $\varphi\in \mathrm{Aut}(X)$ of finite order such that 
\[
F(\lambda)(t) = \lambda(F(t))= q\varphi (\lambda)(t), \qquad \textrm{for all} \,\, t\in \tT. 
\]
In other words, $F$ acts on $X$ as the automorphism $q\varphi$. The rational structure of $\bG$ is encoded in the automorphism $F=q\varphi$.
In particular, $\bG$ is split if and only if $\varphi$ is trivial.

\subsection{Complete root datum} 
Let $W$ be the Weyl group of $\Phi$. For each $\alpha \in \Phi$, let $\alpha^\vee\in \Phi^\vee$ be the corresponding coroot. Define $s_\alpha \colon X\ra X$ by 
 \[
s_\alpha(x) \colonequals x-\langle x,\alpha^\vee\rangle \alpha, \qquad \forall x\in X. 
 \]
The map $\alpha \mapsto s_\alpha$ defines an embedding  $W\hookrightarrow \Aut(X)\subseteq \GL(X_{\mathbb{R}})$.  Consider the coset 
\[
\varphi W = \{\varphi\circ w \, | \, w\in W\} \subseteq \GL(X_{\mathbb{R}}).
\]
\begin{defe} 
Following \cite{GeckMalle}, 
we call $\widehat{\Psi} \colonequals (X,X^\vee, \Phi, \Phi^\vee, \varphi W)$ the \emph{complete root datum} of $\bG$. Similarly, we call $\widehat{\Phi} \colonequals (\Phi, \varphi W)$ the complete root system of $\bG$. 
\end{defe}  
The advantage of the complete root datum and root system is that $q$ does not appear in their definition. Given a complete root datum $\widehat{\Psi}$, for every prime power $q$, we have a unique, up to isomorphism, connected reductive group $\bG$ over $\Fq$ whose complete root datum is $\widehat{\Psi}$. We call $\bG$ the \emph{realisation} of $\widehat{\Psi}$ over $\Fq$.

\subsubsection{Dual group} 
Let $\bG^\vee$ be the group over $\Fq$ dual to $\bG$. By definition, this is the connected reductive group over $\Fq$ whose complete root datum  is given by $(X^\vee, X, \Phi^\vee, \Phi, \varphi^\vee W)$, where $ \varphi^\vee$ is the transpose of $\varphi$.

\subsubsection{Frobenius action on $W$} The action of the Frobenius on $\tG$ stabilises $\tT$ and $N_{\tG}(\tT)$. Thus, $F$ acts on $W=N_{\tG}(\tT)/\tT$. We denote the resulting automorphism of $W$ by $\sigma$.
We call $\widehat{W} \colonequals (W, \sigma)$ the \emph{complete Weyl group}. Elements $w_1$ and $w_2$ in $W$ are said to be \emph{$\sigma$-conjugate} if there exists $w\in W$ such that $w  w_1 \sigma(w)^{-1}=w_2$. If $\bG$ is split, $\sigma$-conjugacy is just the usual conjugacy.

\subsection{Finite reductive groups} 
Let $\bG$ be a connected reductive group over $\Fq$. The finite group  $\bG(\Fq)=\tG^F$ is called a  \emph{finite reductive group}.  Note that this  definition excludes Suzuki and Ree groups.

\subsubsection{Order polynomial} 
Let
\begin{equation} \label{eq:orderPolynomial} 
|\!|\bG|\!|(t) \colonequals t^{|\Phi^+|}  \det (t { \cdot} \mathrm{id}_X - \varphi^{-1})  \sum_{w\in W^\sigma} t^{l(w)}\in \bbZ[t]. 
\end{equation} 
Then  $|\bG(\Fq)|=|\!|\bG|\!|(q)$ \cite{GeckMalle}*{Remark 1.6.15}. Observe that this equality may not hold if we replace $q$ by $q^n$. In other words, $\bG$ may be not polynomial count.  It is, however, polynomial count if we assume that $\bG$ is split, in which case, the counting polynomial simplifies to  
  \begin{equation} \label{eq:orderPoly}
|\!|\bG|\!|(t)=t^{|\Phi^+|}(t-1)^{\rank(X)} \sum_{w\in W} t^{l(w)}.
\end{equation}  

\section{Genus numbers are PORC}\label{s:genus}
The aim of this section is to state a theorem of Deriziotis \cite{Deriziotis85} which tells us that genus numbers for finite reductive groups are PORC. We start by recalling the definition of the genus of a semisimple element due to Carter \cite{Carter1978}.

\subsection{Genus map} 
 Let $\bG$ be a connected reductive group over $\Fq$ and $\bG(\Fq)^{\mathrm{ss}}$ the set of semisimple elements of $\bG(\Fq)$. For each $x\in \bG(\Fq)^{\mathrm{ss}}$,  let $\bG_x$ denote its centraliser in $\bG$. 
 It is well-known that $\bG_x$ is a (possibly disconnected) maximal rank reductive subgroup of $\bG$. 
Thus, the root system of $\bG_x^\circ$ is a closed subsystem $\Phi_1\subseteq \Phi$. We now explain how to encode the rational structure of $\bG_x^\circ$ in root theoretic terms. 

\subsubsection{} Let $W_1\subseteq W$ be the Weyl group of $\Phi_1$ and $N_W(W_1)$ the normaliser of $W_1$ in $W$. Let $(W,\sigma)$ be the complete Weyl group of $\bG$. Note that the action of $\sigma$ on $W$ stabilises $W_1$. Thus, $\sigma$ acts on $N_W(W_1)/W_1$ and we have the notion of $\sigma$-conjugacy for this group. By a theorem of Carter \cite{Carter1978}*{\S 2}, the rational structure of $\bG_1$ is encoded in a $\sigma$-conjugacy class of $N_W(W_1)/W_1$. 

\begin{defe} Let $\Xi(\widehat{\Phi})$ denote  the set of pairs $\xi=([\Phi_1], [w])$ consisting of a $W$-orbit of a closed subsystem $\Phi_1\subseteq \Phi$ and a $\sigma$-conjugacy class $[w]\subseteq N_W(W_1)/W_1$. We refer to $\xi$ as a \emph{genus}  and call $\Xi(\widehat{\Phi})$ the \emph{set of genera} of $\widehat{\Phi}$. If $\bG$ is split, then the complete root datum is just the same as the root datum so we denote this set by $\Xi(\Phi)$. 
\end{defe}

\subsubsection{} Let $\bG^{\mathrm{[ss]}}(\Fq)=\bG^{\mathrm{ss}}(\Fq)/\bG(\Fq)$ denote the set of semisimple conjugacy classes of $\bG(\Fq)$. The above discussion implies that we have a canonical map, called the \emph{genus map}, 
\begin{align*}
\alpha_{\bG(\Fq)} \colon \bG^{\mathrm{[ss]}}(\Fq) &\longrightarrow \Xi(\widehat{\Phi}) \\
 x &\longmapsto [\bG_x^\circ],
\end{align*}
which sends a semisimple conjugacy class to its genus. The number of points of fibres of this map are known as \emph{genus numbers}.

\subsection{Genus numbers} Let $\widehat{\Psi}=(X, X^\vee, \Phi, \Phi^\vee, \varphi W)$ be a complete root datum. 
For each genus $\xi \in \Xi(\widehat{\Phi})$, we define a map $G_\xi^{\mathrm{[ss]}}$ from finite fields to sets as follows:
 Given a finite field $\Fq$, let $\bG$ be the realisation of $\widehat{\Psi}$ over $\Fq$ and set 
 \[
G_\xi^{\mathrm{[ss]}}(\Fq) \colonequals \{x\in \bG^{\mathrm{[ss]}}(\Fq) \, | \, \alpha_{\bG(\Fq)}(x)=\xi\}. 
 \]
Let $d(\Psi)$ denote the modulus as in Definition \ref{d:modulus}.

\begin{thm}[\cite{Deriziotis85}] If $X^\vee/\langle \Phi^\vee\rangle$ is free, then $G_\xi^{\mathrm{[ss]}}$ is PORC count with modulus $d(\Psi)$. \label{t:Deriziotis} 
\end{thm}

\subsubsection{}  The freeness assumption implies that every realisation $\bG$ of $\widehat{\Psi}$ has simply connected derived subgroup. A theorem of Steinberg then implies that centralisers of semisimple elements of $\bG$ are connected. 
As shown in \cite{Deriziotis85}, we have 
\begin{equation}\label{eq:degree} 
 \deg |\!|G_\xi^{\mathrm{[ss]}}|\!|_i= \rank \, G - \rank \langle \Phi_1 \rangle. 
\end{equation} 

\subsubsection{Example}  Let $G=\GL_n$ and $\xi=(\emptyset, [1])$. The reductive subgroup of maximal rank associated to $\xi$ is the diagonal torus. 
 Thus, $G_\xi^{\mathrm{[ss]}}(\Fq)$ is the set of regular diagonal elements in $G(\Fq)$, up to permutation. Hence, 
\[
\displaystyle |\!|G_\xi^{\mathrm{[ss]}}|\!|(t)={t-1 \choose n}=\frac{(t-1)(t-2)\cdots (t-n-1)}{n!}. 
\]
 Note that 
$|\!|G_\xi^{\mathrm{[ss]}}|\!|(q)=0$ for all $q\leq n+1$. This is just the re-statement of the fact that if $q\leq n+1$, there are no regular diagonal elements. 

\subsubsection{} 
The polynomials $|\!|G^{\mathrm{[ss]}}_\xi|\!|_i$ have been determined explicitly for all genera $\xi$ of groups of exceptional type or type $A$, and for many genera of groups of type $B$, $C$, or $D$. However, as far as we understand, this problem has not been fully solved; see \cite{Fleischmann97} for further details.

\section{Representations of finite reductive groups} \label{s:Lusztig} Let $\bG$ be a connected reductive group over $\Fq$.  
In this section, we recall deep results of Lusztig on the structure of complex representations of $\bG(\Fq)$.

\subsection{Unipotent representations} Let   $\Irr_u(\bG(\Fq))$ denote  the set of (irreducible) complex unipotent characters.  
\begin{thm}[\cites{Lusztig84, Lusztig93}] There exists a finite set $\mathfrak{U}(\widehat{W})$, depending only on $\widehat{W}$, together with a function 
\begin{align*}
\mathrm{Deg} \colon \mathfrak{U}(\widehat{W}) &\longrightarrow \bbZ[|W|^{-1}][t] \\
 \rho &\longmapsto \mathrm{Deg}(\rho), 
\end{align*}
 such that the following holds: 
 \begin{itemize} 
 \item[] We have a bijection 
 \[
  \mathfrak{U}(\widehat{W})\longleftrightarrow \Irr_u(\bG(\Fq))
  \] 
  such that
 $\mathrm{Deg}(\rho)(q)$ is the degree of the unipotent character of $\bG(\Fq)$ associated to $\rho$.  \end{itemize} 
 \label{t:unipotent}
\end{thm}

\subsubsection{Remarks} 
\begin{enumerate} 
\item[(i)] The pair $(\mathfrak{U}(\widehat{W}), \mathrm{Deg})$ has been determined explicitly by Lusztig in all types, cf. the appendix of  \cite{Lusztig84}. 
\item[(ii)] If $(W, \sigma)=(S_n, \mathrm{id})$, then $\mathfrak{U}(W)$ equals $\mathscr{P}_n$ the set of partitions of $n$. Moreover, if $\lambda=(\lambda_1, \dots,\lambda_m)$ is a partition of $n$ with $\lambda_1 \leq \lambda_2 \leq \cdots \leq \lambda_m$ 
and $\rho_\lambda(q)$ is the corresponding unipotent representation of $\bG(\Fq)$, then  
\[
\displaystyle \mathrm{Deg} \, \rho_\lambda (q) = \frac{(q-1) |\bG(\Fq)|_{p'} \prod\limits_{1 \leq i < j \leq m} (q^{\alpha_j} - q^{\alpha_i})}{q^{{m -1 \choose 2} + {m-2 \choose 2} + \cdots} \prod\limits_{i=1}^m \prod\limits_{k=1}^{\alpha_i} (q^k -1)},
\]
where $\alpha_i \colonequals \lambda_i + (i-1)$, and $|\bG(\Fq)|_{p'}$ denotes the prime to $p$ part of $|\bG(\Fq)|$. 
\item[(iii)] One knows that the polynomial $\mathrm{Deg}(\rho)$ divides the order polynomial $|\!|\bG|\!|$, cf. \cite{GeckMalle}*{Rem. 2.3.27}. Thus, $|\!|\bG|\!|/\mathrm{Deg}(\rho)$ is a polynomial in $\bbZ[t]$. 
\end{enumerate}

\subsection{Jordan decomposition of characters}\label{ss:LusztigJordan} 
Let $\bG^\vee$ denote the dual group of $\bG$ over $\Fq$. 

\begin{thm}[Lusztig's Jordan Decomposition \cite{Lusztig84}*{Thm. 4.23}] Suppose $\bG$ has connected centre. 
Then we have a bijection
\[
\Irr(\bG(\Fq)) \longleftrightarrow \bigsqcup_{[x] \in \bG^{\vee,{\mathrm{[ss]}}}(\Fq)}  \Irr_u(\bG^\vee_x(\Fq)).
\]
  Moreover, if $\chi\in \Irr(\bG(\Fq))$ is matched with $\rho\in \Irr_u(\bG^\vee_x(\Fq))$, then 
\[
\chi(1) = \rho(1) [\bG^\vee(\Fq):\bG^\vee_x(\Fq)]_{p'},
\]
where the subscript $p'$ (resp. $p$) denotes the prime to $p$ part (resp. the $p$ part). 
 
\label{t:Jordan} 
\end{thm}

\subsubsection{}  Let $r(x) \colonequals |\Phi^+|-|\Phi_x^+|$. Then we have:

\begin{lem} In the above theorem, the relationship between degrees of $\chi$ and $\rho$ can be reformulated as follows: 
\[
\frac{|\bG(\Fq)|}{\chi(1)}=q^{r(x)} \frac{|\bG_x^\vee(\Fq)|}{\rho(1)}.
\]
\end{lem} 

\begin{proof} Indeed, by \ref{eq:orderPoly},
$[\bG^\vee(\Fq) : \bG^\vee_x(\Fq)]_{p} = q^{r(x)}$. Thus, $\displaystyle \chi(1)q^{r(x)}=\rho(1) \frac{|\bG^\vee(\Fq)|}{|\bG_x^\vee(\Fq)|}$. 
The result now follows from the fact that $|\bG(\Fq)|=|\bG^\vee(\Fq)|$.  
\end{proof}

\section{Proofs of main results}\label{s:main}

\subsection{Counting points}\label{ss:counting}
 In this subsection, we prove Theorem \ref{t:main}. Recall that $G$ is a connected split reductive group over $\bbZ$. 
\begin{enumerate} 
\item 
As mentioned in \S \ref{sss:Frobenius}, Frobenius' theorem  implies 
\[
|\fX(\Fq)|=\sum_{\chi\in \Irr(G(\Fq))} \left(\frac{|G(\Fq)|}{\chi(1)}\right)^{2g-2}.  
\]

\item As $G$ has connected centre, Lusztig's Jordan decomposition (Theorem \ref{t:Jordan}) implies 
\[ 
|\fX(\Fq)|=\sum_{[x]\in G^\vee(\Fq)^{\mathrm{ss}}/G^\vee(\Fq)}\quad 
\sum_{\rho\in \Irr_u(G_x^\vee(\Fq))} q^{r(x)(2g-2)} \left(\frac{|G_x^\vee(\Fq)|}{\rho(1)}\right)^{2g-2}.  
\]

\item 
Lusztig's classification of unipotent representations (Theorem \ref{t:unipotent}) then gives 
\[
|\fX(\Fq)|=\sum_{[x]\in G^\vee(\Fq)^{\mathrm{ss}}/G(\Fq)}\, \, q^{r(x)(2g-2)}
\sum_{\rho\in \mathfrak{U}(\widehat{W_x})} \left(\frac{|\!|G_x^\vee|\!|}{\mathrm{Deg}_x(\rho)}\right)^{2g-2}(q).  
\] 
Here $\mathrm{Deg}_x$ denotes the degree function associated to the (possibly non-split) connected reductive group $G^\vee_x$; see Theorem \ref{t:unipotent}.

\item Using the notion of genus (\S \ref{s:genus}) for $G^\vee$, we can re-write the above sum as 
\[ 
|\fX(\Fq)|=\sum_{\xi \in \Xi({\Phi^\vee})} \quad 
\sum_{[x]\in G_\xi^\vee(\Fq)}\, \, q^{r(x)(2g-2)}
\sum_{\rho\in \mathfrak{U}(\widehat{W_x})} \left(\frac{|\!|G_x^\vee|\!|}{\mathrm{Deg}_x(\rho)}\right)^{2g-2}(q).  
\]
\item Observe that $r(x)$, $|\!|G_x^\vee|\!|$, $\mathfrak{U}(\widehat{W_x})$, and $\mathrm{Deg}_x(\rho)$ depend only on the genus $\xi$ of the semisimple class $[x]$. Thus, we may denote these by $r(\xi)$, $|\!|G_\xi^\vee|\!|$, $\mathfrak{U}(\widehat{W_\xi})$, and $\mathrm{Deg}_\xi(\rho)$. Hence, we can re-write the above sum as 
\[
|\fX(\Fq)|=\sum_{\xi \in \Xi({\Phi^\vee})} q^{r(\xi)(2g-2)}
\sum_{\rho\in \mathfrak{U}(\widehat{W_\xi})} \left(\frac{|\!|G_\xi^\vee|\!|}{\mathrm{Deg}_\xi(\rho)}\right)^{2g-2}(q)  
\quad 
\sum_{[x]\in G_\xi^\vee(\Fq)} 1. 
\]

\item Recall the definition $d=d(G^\vee)$ from Theorem \ref{t:Deriziotis}. Assume that $q\equiv i \mod d$, where $i\in \{0,1,\cdots, d(G^\vee)\}$. Then Theorem \ref{t:Deriziotis} gives 
\[
|\fX(\Fq)|=
\sum_{\xi \in \Xi({\Phi^\vee})} q^{r(\xi)(2g-2)}
\sum_{\rho\in \mathfrak{U}(\widehat{W_\xi})} \left(\frac{|\!|G_\xi^\vee|\!|}{\mathrm{Deg}_\xi(\rho)}\right)^{2g-2}(q) |\!|G_\xi^{\mathrm{\vee, [ss]}}|\!|_i(q).
\]
\end{enumerate}

\begin{defe}  \label{d:main} 
For $i\in \{0,1,\cdots, d(G^\vee)\}$, define polynomials 
\begin{equation} 
|\!|\fX|\!|_i = 
\sum_{\xi \in \Xi({\Phi^\vee})} 
t^{r(\xi)(2g-2)} |\!|G_\xi^{\mathrm{\vee, [ss]}}|\!|_i
\sum_{\rho\in \mathfrak{U}(\widehat{W_\xi})}
\left(\frac{|\!|G_\xi^\vee|\!|}{\mathrm{Deg}_\xi(\rho)}\right)^{2g-2}  \in \bQ[t]. 
\end{equation} 
\end{defe} 
Note that each summand is polynomial with rational coefficients. The sum is over objects which depend only on the complete root datum $\Psi$ i.e., they are independent of $q$. It follows that $|\!|\fX|\!|\in \bbQ[t]$. 
The above discussion then shows that $\fX$ is PORC count with counting polynomials $|\!|\fX|\!|_i$. Thus, Theorem \ref{t:main} is proved. \qed

\subsubsection{Aside on representation $\zeta$-function} \label{sss:zeta}
For each  $i\in \{0,1,\cdots, d(G^\vee)-1\}$ and $u\in \mathbb{C}$, let 
\[
\xi_i(u,t) \colonequals  
\sum_{\xi \in \Xi({\Phi^\vee})} 
t^{r(\xi)u} |\!|G_\xi^{\mathrm{\vee, [ss]}}|\!|_i
\sum_{\rho\in \mathfrak{U}(\widehat{W_\xi})}
\left(\frac{|\!|G_\xi^\vee|\!|}{\mathrm{Deg}_\xi(\rho)}\right)^{u}. 
\]
Then we have equality of complex functions 
\[
\zeta_{G(\Fq)}(s) = \xi_i(s, q),\qquad \forall q\equiv i \mod d(G^\vee). 
\]

\subsection{Counting polynomials in the case $g=1$} 
In this subsection, we show that if $g=1$, then  $|\!|\fX|\!|_i$ has degree $\rank(G)$ and leading coefficient $1$. This establishes Corollary \hyperref[c:main]{\ref*{c:main}(i)}. 

\subsubsection{} By the above discussion, 
 \[
|\!|\fX|\!|_i = 
\sum_{\xi \in \Xi({\Phi^\vee})}|\!|G_\xi^{\mathrm{\vee, [ss]}}|\!|_i |\mathfrak{U}(\widehat{W_\xi})|. 
\]
In view of equation \ref{eq:degree},  the degree of a summand is maximal if and only if $\Phi_1$ is empty; i.e., when $\xi$ is the genus of a \emph{regular} semisimple element. Thus, the degree of $|\fX(\Fq)|$ equals 
$\rank(G)$. 

\subsubsection{} 
Next, one can easily check that for genera $(\emptyset, [w])$, 
the leading coefficient of $|\!|G_\xi^{\mathrm{\vee, [ss]}}|\!|_i$ equals $1/{|W_w|}$, where $W_w$ denotes the centraliser of $w$ in $W$. 
Thus, the leading coefficient of $|\!|\fX|\!|_i$  is $\sum 1/|W_w|$, where the sum runs over conjugacy classes of $W$. By the orbit stabiliser theorem, this sum equals $1$. This establishes Corollary \ref{c:main}(i). 
\qed

 \subsubsection{} Here is an alternative (perhaps more intuitive) argument for this corollary. By \eqref{eq:Frobenius}, $|\fX(\Fq)|$ is the counting polynomial for the number of conjugacy classes of $G(\Fq)$.  Now the leading term of the class number equals the leading term of the polynomial counting semisimple elements. By a theorem of Steinberg, the latter equals $|Z(G(\Fq))|q^{\rank([G,G])}$. Thus, the leading coefficient is $1$ and the degree is $\dim(Z(G))+\rank([G,G])=\rank(G)$.

\subsection{Counting polynomials in the case $g>1$} 
In this subsection, we show that if $g>1$, then the polynomial $|\!|\fX|\!|_i$ has degree $(2g-2)\dim G+\dim Z(G^\vee)$ and leading coefficient $|\pi_0(Z(G^\vee))|$. This establishes Corollary \hyperref[c:main]{\ref*{c:main}(ii)}. 

\subsubsection{} 
We claim that only $\xi=([\Phi^\vee], 1)$ contributes to the leading term of $|\!|\fX|\!|_i$. Note that a semisimple element has genus  $([\Phi^\vee], 1)$ if and only if it is central. 
Thus, the claim implies that the leading term of $|\!|\fX|\!|_i(q)$ is the same as the leading term of $\displaystyle |Z(G^\vee(\Fq))| |G^\vee(\Fq)|^{2g-2}$, which would establish the desired result.\footnote{Aside: The claim amounts to the statement that only $1$-dimensional representations contribute to the leading term of \eqref{eq:Frobenius}. Thus, the leading term of $|\fX(\Fq)|$ equals the leading term of $|G(\Fq)^{\mathrm{ab}}|.|G(\Fq)|^{2g-2}$, where $G(\Fq)^{\mathrm{ab}}$ denotes the abelianisation of $G(\Fq)$.}

\subsubsection{} 
For ease of notation, set $n \colonequals 2g-2$.  Let $\xi=([\Phi_1], [w])\in \Xi({\Phi^\vee})$ be a genus. Thus, $\Phi_1$ denotes a closed subsystem of the dual root system $\Phi^\vee$.  Let  
\[
P_{\xi, n}(t) = t^{n.r(\xi)} |\!|G_\xi^{\mathrm{\vee, [ss]}}|\!|_i
\sum_{\rho\in \mathfrak{U}(\widehat{W_\xi})}
\left(\frac{|\!|G_\xi^\vee|\!|}{\mathrm{Deg}_\xi(\rho)}\right)^{n}.
\]
\subsubsection{} 
Observe that 
\[
\begin{split} 
\deg P_{\xi, n} & =  n.r(\xi) + \deg |\!|G_\xi^{\mathrm{\vee, [ss]}}|\!|_i  + n.\dim(G_\xi^\vee) \\
 & = n\Big(|\Phi^+|-|\Phi_1^+|\Big)+\Big(\rank(X)-\rank\langle \Phi_1\rangle \Big) + n.\dim(G_\xi^\vee)\\
 & = n\Big(|\Phi^+|-|\Phi_1^+|+\dim(G_\xi^\vee)\Big) + \Big(\rank(X)-\rank\langle \Phi_1\rangle\Big ) \\
 & = n\Big(|\Phi^+|-|\Phi_1^+|+2|\Phi_1^+|+\rank(X)\Big) +  \Big(\rank(X^\vee)-\rank\langle \Phi_1\rangle \Big)\\
 & = n\Big(|\Phi^+|+|\Phi_1^+|+\rank(X)\Big) +  \Big(\rank(X^\vee)-\rank\langle \Phi_1\rangle \Big). \\
 \end{split} 
\]
Thus, 
\[
\begin{split} 
n\dim(G)+\dim(Z(G^\vee)) -  \deg P_{\xi, n}  & =  n\Big (2|\Phi^+|+\rank(X)\Big) + \Big(\rank(X^\vee)-\rank \langle \Phi^\vee \rangle\Big) - \deg P_{\xi,n} \\
&= n\Big(|\Phi^+|-|\Phi_1^+|\Big) + \rank \langle  \Phi_1 \rangle -\rank \langle  \Phi^\vee \rangle. \\ 
\end{split}
\]

\subsubsection{} 
It is clear that the above quantity is $0$ if $\Phi_1=\Phi^\vee$  i.e. if $\xi$ is central.  If $\xi$ is not central; i.e. $\Phi_1$ is strictly smaller than $\Phi^\vee$, then the above quantity is positive because
\[
n\Big(|\Phi^+|-|\Phi_1^+|\Big) > \rank \langle  \Phi^\vee \rangle -\rank \langle  \Phi_1 \rangle. 
\]
This follows from the fact that elements of $(\Phi^\vee)^+-\Phi_1^+$ span the vector space $(\langle  \Phi^\vee \rangle/\langle  \Phi_1 \rangle)\otimes \bbR$ and that $n\geq 2$. This concludes the proof.
\qed

\section{Examples: $\PGL_2$, $\PGL_3$, $\mathrm{SO}_5$, and $G_2$}\label{s:Examples} 
In this section, we give tables containing the genera $\xi\in \Xi({\Phi}^\vee)$, the integer $r(\xi)$, the size of centraliser $|G_\xi^\vee(\Fq)|$, genus number $|G_\xi^{\mathrm{\vee, [ss]}}(\Fq)|$, and the unipotent degrees of the centraliser, for simple adjoint groups $G$ of rank $\leq 2$. We assume throughout that $q\equiv 1 \mod d(G^\vee)$. By the discussion of \S \ref{ss:counting}, the counting polynomial $|\!|\fX|\!|_1$ of the associated character stacks can be determined using these tables. 

\subsection{The case $G=\PGL_2$} In this case, $G^\vee=\mathrm{SL}_2$. In this case $d(G^\vee)=2$. 
So let us assume $q$ is odd. Then the genera for $\PGL_2$ is given in Table \ref{tab:PGL2}. 

\begin{table}[h]
\begin{tabular}{|c|c|c|c|c|} 
\hline
$\Xi({\Phi^\vee})$  & $r(\xi)$ & $|G_\xi^\vee(\Fq)|$ & $|G_\xi^{\mathrm{\vee, [ss]}}(\Fq)|$ & Unipotent degrees \\ \hline
$(A_1, 1)$ & 0 & $q(q^2-1)$ & 2 & $1,q$ \\ \hline
$(\emptyset, 1)$ & 1 & $q-1$ & $\frac{q-3}{2}$ & 1 \\ \hline 
$(\emptyset, w)$ & 1  & $q+1$ & $\frac{q-1}{2}$ & 1\\ \hline
\end{tabular} \vspace{0.5em}
\caption {Genera for $\PGL_2$} \label{tab:PGL2} 
\end{table}

\subsubsection{} Using the table, we find 
\begin{align*}
  |\fX(\Fq)| = 2\big((q(q^2-1))^{2g-2} + (q^2-1)^{2g-2}\big) + \frac{q-3}{2} q^{2g-2} (q-1)^{2g-2} + \frac{q-1}{2} q^{2g-2}(q + 1)^{2g-2}.
\end{align*}

\subsubsection{} 
For instance, for $g=2, 3, 4$ we obtain, respectively, the polynomials 
\[ 
2 q^6+q^5-4 q^4+3 q^3-4 q^2+2, 
\] 
\[ 
2 q^{12}-8 q^{10}+q^9+12 q^8+10 q^7-28 q^6+5 q^5+12 q^4-8 q^2+2,
 \] 
\[
2 q^{18}-12 q^{16}+30 q^{14}+q^{13}-40 q^{12}+21 q^{11}-12 q^{10}+35 q^9-12 q^8+7 q^7-40 q^6+30 q^4-12 q^2+2.
 \] 

\subsection{The case $G=\PGL_3$} In this case, $G^\vee=\mathrm{SL}_3$. Then $d(G^\vee)=3$. So let us assume $q \equiv 1 \mod 3$. 
Then the genera is given in Table \ref{tab:PGL3}. 
\begin{table}[h]
\begin{tabular}{|l|l|l|l|l|} 
\hline
$\Xi({\Phi^\vee})$  & $r(\xi)$ & $|G_\xi^\vee(\Fq)|$ & $|G_\xi^{\mathrm{\vee, [ss]}}(\Fq)|$ & Unipotent degrees \\ \hline
$(A_2, 1)$ & 0 & $q^3(q^3-1)(q^2-1)$ & 3  & $1$, $q(q+1)$, $q^3$ \\ \hline
$(A_1, 1)$ & 2 & $q(q^2-1)$ & $q-4$ & $1,q$\\ \hline 
$(\emptyset, 1)$ & 6  &   $(q-1)^2$ & $\frac16(q^2-5q+10)$ & 1  \\ \hline 
$(\emptyset,w_1)$ & 6  & $q^2-1$ & $\frac12 q(q-1)$ & 1 \\ \hline
$(\emptyset, w_2)$ & 6  & $q^2 + q + 1$ & $\frac13 (q^2+q-2)$  & 1\\ \hline
\end{tabular} \vspace{0.5em}
\caption {Genera for $\PGL_3$. Here $w_1$ and $w_2$ are simple generators for $W$. } \label{tab:PGL3} 
\end{table}

\subsubsection{} Using the table,  we find: 
\begin{align*}
|\fX(\Fq)| &= 3 \Big(\big(q^3(q^3-1)(q^2-1)\big)^{2g-2} + \big(q^2(q^3-1)(q^2-1)\big)^{2g-2} + \big(q^2(q^3-1)(q-1)\big)^{2g-2}\Big) \\
&+ (q-4)q^{4g-4}\Big((q(q^2-1))^{2g-2} + (q^2-1)^{2g-2}\Big) \\
&+ q^{12g-12} \bigg( \frac{q^2-5q+10}{6}( (q-1)^{2})^{2g-2} + \frac{q(q-1)}{2}(q^2-1)^{2g-2} \\
&+ \frac{q^2+q-2}{3}(q^2+q+1)^{2g-2}   \bigg).
\end{align*}

\subsubsection{}
For instance, for $g=2, 3$ we obtain, respectively, the polynomials 
\[ 
3 q^{16}-6 q^{14}-5 q^{13}+q^{12}+13 q^{11}+17 q^{10}-33 q^9+23 q^8-29 q^7+8 q^6+15 q^5+2 q^4-6 q^3-6 q^2+3
 \]
 and
\begin{align*}
&3 q^{32}-12 q^{30}-12 q^{29}+18 q^{28}+48 q^{27}+6 q^{26}-71 q^{25}-74 q^{24}+42 q^{23}+131 q^{22} \\
&-53 q^{21}+52 q^{20}-104 q^{19}+57 q^{18}-261 q^{17}+446 q^{16}-156 q^{15}-23 q^{14}-129 q^{13} \\
&+62 q^{12}-34 q^{11}+114 q^{10}+41 q^9-70 q^8-72 q^7+6 q^6+48 q^5+18 q^4-12 q^3-12 q^2+3. 
\end{align*}

\subsection{The case $G=\mathrm{SO}_5$} In this case, $G^\vee=\mathrm{Sp}_4$ and $d(G^\vee)=2$. So assume $q$ is odd. Table \ref{tab:SO5} gives the genera and the invariants required for writing an explicit description for $|\fX(\Fq)|$. 
\begin{table}[h]
\begin{tabular}{|l|l>{\raggedright}p{0.2\linewidth} |>{\raggedright}p{0.15\linewidth} |>{\raggedright\arraybackslash}p{0.25\linewidth}|} 
\hline
$\Xi({\Phi^\vee})$  & $r(\xi)$ & $|G_\xi^\vee(\Fq)|$ & $|G_\xi^{\mathrm{\vee, [ss]}}(\Fq)|$ & Unipotent degrees \\ \hline
$(B_2, 1)$ & 0 & $q^4(q^4-1)(q^2-1)$ & 2 & $1$, $\frac{q(q^2+1)}{2}$,  $\frac{q(q+1)^2}{2}$,  $\frac{q(q^2+1)}{2}$,  $\frac{q(q-1)^2}{2}$, $q^4$ \\\hline
$(A_1 \times A_1, 1)$ & 2 & $q^2(q^2-1)^2$ & 1 & $1, q, q, q^2$ \\\hline
$(A_1 \times A_1, w_2 w_1 w_2)$ & - & - & -  & -  \\ \hline
$(A_1, 1)$ & 3 & $q(q-1)(q^2-1)$ & $q-3$ & $1,q$ \\ \hline
$(A_1, w_2w_1w_2)$ & 3  & $q(q+1)(q^2-1)$ & $q-1$ & $1, q$  \\ \hline
$(\widetilde{A}_1, 1)$ &  3 & $q(q-1)(q^2-1)$ & $\frac{q-3}{2}$ & $1, q$ \\\hline
$(\widetilde{A}_1, w_1w_2w_1)$ & 3 & $q(q+1)(q^2-1)$ & $\frac{q-1}{2}$ &$1, q$  \\\hline
$(\emptyset, 1)$ & 4 & $(q-1)^2$ & $\frac18 (q-3)(q-5)$ & 1 \\\hline
$(\emptyset, -1)$ &  4 & $(q+1)^2$ & $\frac18 (q-1)(q-3)$ & 1\\\hline
$(\emptyset, [w_1])$ &4  & $q^2-1$ & $\frac14 (q-1)(q-3)$ & 1 \\\hline
$(\emptyset, [w_2])$ &4  & $q^2-1$ & $\frac14 (q-1)^2$ & 1 \\\hline
$(\emptyset, [w_1 w_2])$ &  4 & $q^2 + 1$ & $\frac14 (q^2-1)$ & 1 \\\hline
\end{tabular} \vspace{0.5em}
\caption {Genera for $\mathrm{SO}_5$. Here the two copies of $A_1$ inside $C_2$ give rise to non-conjugate centralisers, so one of the copies is denoted by $\widetilde{A}_1$. The twisted $A_1\times A_1$ is a reductive subgroup of $G^\vee$ but does not arise as a centraliser. } \label{tab:SO5} 
\end{table}

\subsubsection{}
If $g>1$, then 
in the polynomial $|\fX(\Fq)|$, there is a single term not divisible by $q-1$, namely, 
\[
2\Big(2(q^3(q^3+q^2+q+1)(q+1))\Big)^{2g-2}.
\]
 Thus, if we plug in $q=1$ in $|\fX(\Fq)|$, we obtain $2^{8g-7}$. \\

\subsection{The case $G=G_2$} In this case, $G^\vee$ also equals $G_2$. Then $d(G^\vee)=6$. So assume $q\equiv 1 \mod 6$. The counting polynomial $|\fX(\Fq)|$ can be obtained using the genera listed in Table \ref{tab:G2}. 

\begin{table}[h]
\begin{tabular}{|l|l|>{\raggedright}p{0.2\linewidth} |>{\raggedright}p{0.13\linewidth} |>{\raggedright\arraybackslash}p{0.2\linewidth}|}
\hline
$\Xi({\Phi^\vee})$  & $r(\xi)$ & $|G_\xi^\vee(\Fq)|$ & $|G_\xi^{\mathrm{\vee, [ss]}}(\Fq)|$ & Unipotent degrees \\ \hline
$(G_2, 1)$ & 0 & $q^6(q^6-1)(q^2-1)$ & 1 & $1,  \frac16 q \Phi_2^2\Phi_3$,  $\frac16 q \Phi_1^2 \Phi_6,  \frac12 q \Phi_2^2 \Phi_6$,  $\frac13 q \Phi_3 \Phi_6,  \frac13 q \Phi_3 \Phi_6$,  $\frac12 q \Phi_1^2 \Phi_3,  \frac13 q \Phi_1^2 \Phi_2^2$, $\frac13 q \Phi_1^2 \Phi_2^2, q^6$ \\ \hline
$(A_2, 1)$ & 3 & $q^3(q^3-1)(q^2-1)$ &1  & $1, q(q+1), q^3$ \\ \hline
$(A_1 \times \widetilde{A}_1, 1)$ & 4 & $q^2(q^2-1)^2$ & 1 & $1,q,q,q^2$ \\\hline
$(A_1, 1)$ & 5 & $q(q-1)^2(q+1)$ & $\frac{q-5}{2}$ & $1,q$ \\ \hline
$(A_1, w_1)$ & 5 & $q(q-1)(q+1)^2$  & $\frac{q-1}{2}$ & $1,q$  \\\hline
$(\widetilde{A}_1, 1)$ &5  & $q(q-1)^2(q+1)$ & $\frac{q-3}{2}$ &$1,q$  \\ \hline
$(\widetilde{A}_1, w_2(w_1w_2)^2)$ & 5 & $q(q-1)(q+1)^2$  & $\frac{q-1}{2}$ &$1,q$  \\\hline
$(\emptyset, 1)$ & 6 & $(q-1)^2$  & $\frac{q^2-8q +19}{12}$ & 1\\\hline
$(\emptyset, -1)$ & 6 & $(q+1)^2$ &  $\frac{q^2-4q+3}{12}$ & 1  \\\hline
$(\emptyset, [w_1])$ & 6 & $q^2-1$  & $\frac{(q-1)^2}{4}$ & 1 \\\hline
$(\emptyset, [w_2])$ & 6 & $q^2-1$ & $\frac{(q-1)^2}{4}$ & 1  \\\hline
$(\emptyset, [(w_1w_2)^2])$ &6  & $q^2 + q +1$  & $\frac{q^2+q-2}{6}$ & 1 \\\hline
$(\emptyset, [w_1w_2])$ & 6  & $q^2 - q + 1$  & $\frac{q(q-1)}{6}$  &  1 \\\hline
\end{tabular} 
\vspace{0.5em}
\caption {Genera for $G_2$. Here $\Phi_i$ is the $i$th cyclotomic polynomial; thus,
$\Phi_1=q-1,\qquad \Phi_2=q+1,\qquad \Phi_3=q^2+q+1,\qquad \Phi_6=q^2-q+1$.  
} \label{tab:G2} 
\end{table}

\subsubsection{} 
If $g>1$, then in the polynomial $|\fX(\Fq)|$, there are four terms not divisible by $q-1$, namely 
\[
\Bigg(6\frac{|G(\Fq)|}{q\Phi_1^2\Phi_6}\Bigg)^{2g-2} + 
\Bigg(2\frac{|G(\Fq)|}{q\Phi_1^2\Phi_3}\Bigg)^{2g-2} +
\Bigg(3\frac{|G(\Fq)|}{q\Phi_1^2\Phi_2^2}\Bigg)^{2g-2}  +
\Bigg(3\frac{|G(\Fq)|}{q\Phi_1^2\Phi_2^2}\Bigg)^{2g-2}.
\]
If we set $q=1$ in the above polynomial, we obtain 
\[
(6\cdot 12)^{2g-2} + \Big(2 \cdot \frac{12}{3}\Big)^{2g-2}+ \Big(3 \cdot \frac{12}{4}\Big)^{2g-2} + \Big(3\cdot \frac{12}{4}\Big)^{2g-2} = 72^{2g-2} + 8^{2g-2} + 
2\cdot 9^{2g-2}. 
\]

\section{The character stack of $\GL_n$ and $\PGL_n$ revisited}\label{s:GLn} 
Let $G=\GL_n$ and $g$ a positive integer. Let $\fX$ be the character stack of associated to $(\Gamma_g, G)$. In this section, we give an explicit expression for the number of points of $\fX$ using the same method employed in \cite{HRV}.  The main point is that we have a good direct understanding of character degrees of $G(\Fq)$ without having to resort to Lusztig's Jordan decomposition. It would be interesting to prove directly that the polynomial obtained in this section (see \eqref{eq:GL_n}) equals the one from Definition \ref{d:main}.

\subsection{Conjugacy classes of $\GL_n(\Fq)$} 
Let $\cI=\cI(q)$ denote the set of irreducible polynomials over $\Fq$, except that we exclude $f(t)=t$. Let $\cP$ denote the set of partitions. Let $\cP_n(\cI)$ denote the set of maps  
$\Lambda \colon \cI \ra \cP$
such that 
\[
|\Lambda| \colonequals \displaystyle \sum_{f \in \cI} |\Lambda(f)|\deg(f) =n.
\] 
Then we have a bijection between $\cP_n(\cI)$ and conjugacy classes of $G(\Fq)$. 
 Let 
 \[
 \cP(\cI)=\bigcup_{n\geq 1} \cP_n(\cI).
 \]

\subsubsection{Types} 
Let $\cI_d\subset \cI$ denote the subset of irreducible polynomials of degree $d$ over $\Fq$.
 Given $\Lambda \in \cP(\cI)$ we define
\[
m_{d,\lambda} \colonequals\#\{ f \in \cI_d \, | \, \Lambda(f) =\lambda\}.
\]
The collection of integers $(m_{d,\lambda})$ is called the \emph{type} of $\Lambda$ and is denoted by $\tau=\tau(\Lambda)$. 
\begin{rem} 
One can show that $\Lambda$ and $\Lambda'$ have the same type if and only if the centraliser of the corresponding conjugacy classes in $G(\Fq)$ have the same genus. Thus, we have a bijection between semisimple types and genera of $G$. 
\end{rem}

\subsubsection{} 
The weight of a type $\tau$ is defined by 
\[
|\tau| \colonequals \sum d |\lambda| m_{d,\lambda}=n. 
\]
Thus, the weight of an element $\Lambda \in \cP(\cI)$ equals the weight of its type.

\subsubsection{Genus number} Let $A_\tau(q)$ denote the number of $\Lambda\in \cP(\cI(q))$ of type $\tau$. (Equivalently, $A_\tau(q)$ is the genus number of $\tau$.) Our aim is to give an explicit formula for $A_\tau(q)$. 
Let
\[
T(d) \colonequals \sum_{\lambda \in \cP} m_{d,\lambda}. 
\]
Let $I_d=I_d(q)=|\cI_d|$ denote the number of irreducible polynomials of degree $d$ over $q$. By a result attributed to Gauss, we have 
\begin{equation}\label{eq:Gauss}
 I_d(q) =  
 \begin{cases} 
 q-1 & \textrm{if $d=1$} \\\\
 \displaystyle \frac{1}{d} \sum_{k|d} \mu(k) q^{d/k} & \textrm{otherwise}.
 \end{cases} 
\end{equation}

\begin{lem}\label{l:genusNumber}
$\displaystyle
A_\tau(q)
= {\displaystyle \prod_{d\geq 1} \frac{\displaystyle \prod_{i=0}^{T(d)-1} (I_d(q)-i)}{\displaystyle \prod_{\lambda \in \cP} m_{d,\lambda}!}}.$
 \end{lem} 
 By convention, if $T(d)=0$, then the product involving $T(d)$ is defined to be one. We leave the above lemma as an exercise. As a corollary, we conclude that $A_\tau(q)$ is a polynomial in $q$ with rational coefficients. 

\subsection{Irreducible characters of $\GL_n(\Fq)$} We have seen that irreducible characters of $G(\Fq)$ are in bijection with $\cP_n(\cI)$. 
Let $\chi_\Lambda$ denote the irreducible character of $G$ corresponding to $\Lambda\in \cP_n(\cI)$.

\subsubsection{} Define the normalised hook polynomial associated to the partition $\lambda \in \cP$ by 
\begin{equation} 
H_\lambda(q) \colonequals q^{-\frac{1}{2}\langle \lambda, \lambda \rangle} \prod (1-q^h).
\end{equation}
Here the product is taken over the boxes in the Young diagram of $\lambda$ and $h$ is the hook length of the box. Moreover, $\displaystyle \langle \lambda, \lambda \rangle \colonequals \sum_i (\lambda_i')^2$ where the sum is taken over the parts in the conjugate partition $\lambda'$.

\subsubsection{} Next, define the normalised hook polynomial of $\Lambda\in \cP_n(\cI)$ by
\begin{equation} 
H_\Lambda(q) \colonequals (-1)^n q^{\frac{1}{2} n^2} \prod_{f \in \cI} \left(H_{\Lambda(\lambda)}(q^{\deg f }) \right).
\end{equation}
It is easy to see that $H_\Lambda$ is a monic polynomial in $\mathbb{Z}[q]$. 

\subsubsection{} Let $\Lambda'$ be the map conjugate to $\Lambda$; i.e., $\Lambda'(f)$ is the partition conjugate to $\Lambda(f)$ for all $f\in \cI$. Then by a theorem of Green, we have 
\begin{equation} 
\displaystyle \frac{|G(\Fq)|}{\chi_\Lambda(1)}=H_{\Lambda'}(q). 
\end{equation}  

\subsubsection{} It is clear that the hook polynomial $H_\Lambda(q)$ depends only on the type of $\Lambda$; in fact, we have 
\[
H_\Lambda(q)=(-1)^n q^{\frac{1}{2} n^2} \prod_{d,\lambda} \left(H_{\lambda}(q^d)\right)^{m_{d,\lambda}}.
\]
Given a type $\tau$,   we write $H_\tau(q)$ for the hook polynomial associated to $\tau$.

\subsection{Counting points on $\fX$} Let
\begin{equation} \label{eq:GL_n}
|\!|\fX|\!| \colonequals \sum_{\tau \in \cT_n} A_\tau \big(H_\tau(t)\big)^{2g-2} \in \bbQ[t].
\end{equation} 

From the above discussions, one easily concludes: 

\begin{thm}\label{t:mainGL}  For every finite field $\Fq$, we have $|\fX(\Fq)|=|\!|\fX|\!|(q)$. 
\end{thm} 

\subsubsection{} 
As an example, we consider the case $G=\GL_2$. The types of weight $2$ and their associated invariants are listed in the following table: 
 \[
 \begin{array}{|c|c|c|c|c|c|}\hline
\textrm{Representations} & m_{d,\lambda} & T(d) &  A_\tau & H_\lambda & H_\Lambda \\\hline
\textrm{Cuspidal} & m_{2,(1)}=1&  1& (q^2-q)/2 &  q^{-1/2} (1-q) & q^2q^{-1}(1-q^2) \\\hline
\textrm{Steinberg} & m_{1,(2)}=1& 1& q-1 &  q^{-1} (1-q)(1-q^2) & q^2q^{-1}(1-q^2)(1-q) \\\hline
\textrm{Determinant} & m_{1,(11)}=1& 1&  q-1 &  q^{-2} (1-q^2)(1-q) & q^2q^{-2}(1-q^2)(1-q)\\\hline
\textrm{Principle series} & m_{1,(1)}=2& 2&  (q-1)(q-2)/2 &  q^{-1/2} (1-q) & q^2\left(q^{-1/2} (1-q)\right)^2\\\hline
 \end{array}
 \]  
 Thus, we find:
 \begin{align*}
 |\fX(\Fq)| &= \frac{(q^2-q)}{2} \big(q-q^3\big)^{2g-2} + (q-1) \big(q(1-q)(1-q^2)\big)^{2g-2} \\
 &+ (q-1) \big((1-q^2)(1-q)\big)^{2g-2} + \frac{(q-1)(q-2)}{2} \big(q(1-q)^2\big)^{2g-2}.
 \end{align*}

For instance, for $g=2,3$ we obtain, respectively, the polynomials 
\[ 
q^9-2 q^8-2 q^7+11 q^6-18 q^5+17 q^4-8 q^3-q^2+3 q-1. 
\]
and
\begin{align*}
&q^{17}-5 q^{16}+6 q^{15}+11 q^{14}-34 q^{13}+29 q^{12}-34 q^{11}+124 q^{10}-230 q^9 \\&+204 q^8-74 q^7-q^6-14 q^5+29 q^4-10 q^3-6 q^2+5 q-1. 
\end{align*}

\subsection{The $\PGL_n$ character stack} 
In this subsection, we study the character stack associated to $(\Gamma_g, \PGL_n)$.
For each  $n^\mathrm{th}$ root of unity $\zeta$, consider 
\[
\fX_\zeta\colonequals (\mathbb{G}_m)^{2g}\backslash \Hom_\zeta(\Gamma, \GL_n)/\mathrm{PGL}_n,
\]
where 
\[
\Hom_\zeta(\Gamma, \GL_n(\bbC)) \colonequals\{x_1,...,x_g, y_1,...,y_g\in \GL_n(\bbC)\, | \, [x_1,y_1]\cdots [x_g,y_g]\zeta=1\}.
\]
Then 
\[
\displaystyle \fX=\bigsqcup_{\zeta \in \mu_n} X_\zeta
\]
 is the decomposition of $\fX$ into its connected components. For primitive roots of unity $\zeta$, the arithmetic geometry of $\fX_\zeta$ was studied in \cite{HRV}. We consider the opposite case, i.e., when $\zeta=1$. 
\subsubsection{Counting points on $X_1$} 
Let 
\[
|\!|\fX_1|\!| \colonequals \frac{1}{(q-1)^{2g-1}} \sum_{\tau \in \cT_n} A_\tau(q) H_{\tau}(q)^{2g-2}=\sum_{\tau\in \cT_n} \frac{A_\tau}{(q-1)} \left(\frac{H_{\tau}(q)}{q-1}\right)^{2g-2}.
\]
Then Theorem \ref{t:mainGL} implies: 

\begin{cor} For every finite field $\Fq$, we have $|\fX_1(\Fq)|=|\!|\fX_1|\!|(q)$. 
\end{cor} 

\subsubsection{Proof of Corollary \ref{c:Euler}.(iii)} To obtain the Euler characteristic of $\fX_1$, we  compute the value of $|\!|{\fX}_1|\!|$ at $1$. Observe that $H_\tau(q)$ is divisible by exactly one factor of $(q-1)$ if and only if the only non-zero $m_{d,\lambda}$ in $\tau$ is $m_{n,(1)}=1$. In this case, $H_\tau=(1-q^n)$. Thus, 
\[
\left(\frac{H_{\tau}(q)}{q-1}\right)(1)  = -n. 
\]
Moreover, we have $A_\tau= I_n(q)$; therefore, 
\[
\frac{A_\tau}{(q-1)}(1)  = I_n'(1) = \frac{1}{n} \sum_{k|n} \mu(k)\frac{n}{k} = \phi(n)/n.
\]
Here the last equality follows from the well-known relation between the M\"{o}bius and Euler functions. We therefore obtain 
\[
|\!|\fX_1|\!|(1) = \sum_{\tau\in \cT_n} \left(\frac{A_\tau}{(q-1)}\right)(1) \Bigg(\left(\frac{H_{\tau}(q)}{q-1}\right)(1)\Bigg)^{2g-2}=\phi(n).n^{2g-3}. 
\]
\qed

\appendix 
\section{Complex character stacks} In this appendix, we discuss the implications of our main theorem for the complex character stack $\fX_\bC \colonequals [\Hom(\Gamma_g, G_\bC)/G_\bC]$. To this end, we first recall a theorem of Katz \cite[Appendix]{HRV} on polynomial count schemes over $\bbC$. 

\subsubsection{} 
Let $Y$ be a separated scheme or (algebraic) stack of finite type over $\bbC$. A spreading out of $Y$ is a pair $(A, Y_A)$ consisting of a subring $A\subset \bbC$, finitely generated as a $\bbZ$-algebra, a scheme $Y_A$ over $A$ and  an identification $Y_A \otimes_A \bbC \simeq Y$.   

\begin{defe} The stack $Y/\bC$ is said to be \emph{polynomial count} if there exists a polynomial $|\!|Y|\!|\in \bC[t]$ and spreading out $(A,Y_A)$ such that for every homomorphism $A\ra \Fq$, we have $|Y_A(\Fq)|=|\!|Y|\!|(q)$.
\end{defe} 

\subsubsection{} If $Y/\bbC$ is a polynomial count scheme, then a theorem of Katz states the $E$-polynomial of $Y$ (defined using its mixed Hodge structure) is given by 
\begin{equation} 
E(Y; x,y)=|\!|Y|\!|(xy).
\end{equation}

\subsubsection{} Applying the above considerations to the situation of interest to us, Theorem \ref{t:main} implies: 
\begin{cor} 
 The complex representation variety $\fR_\bC \colonequals \Hom(\Gamma_g, G_\bC)$ is polynomial count with counting polynomial $|\!|\fR|\!| \colonequals |\!|\fX|\!|_1\times |\!|G|\!|$. Thus, 
\[
E(\fR_\bC; x,y)=|\!|\fR|\!|(xy).
\]
\end{cor} 

\begin{proof} Let $d=d(G^\vee)$ and  $A \colonequals \bbZ[1/d,\zeta_d]$. It is easy to see that we have a unital algebra homomorphism $A\ra \Fq$ if only if $q\equiv 1 \mod d$, cf. \cite[Lemma 3.1]{BaragliaHekmati}. Now let us choose the spreading out $\fR_A$. Then Theorem \ref{t:main} states that for every homomorphism $A\ra \Fq$, we have $|\fR_A(\Fq)|=(|\!|\fX|\!|_1\times |\!|G|\!|)(q)$.
\end{proof} 

\subsubsection{} The analogue of Corollary \ref{c:main} gives us the dimension and number of irreducible components of highest dimension of $\fR_\bC$. For $G=\GL_n$, these results were obtain in \cite{Chernosouv}, who also proved that $\fR_\bC$ is irreducible and rational. On the other hand, for a general semisimple $G$,  it was proved in \cite{Li} that $\pi_0(\fR_\bbC)=|\pi_1([G,G])|$. 

\subsubsection{} Next, suppose $Y$ is an algebraic stack of finite type over $\bbC$.
Thinking of $Y$ as a simplicial scheme (cf. the appendix of \cite{Shende}), we have a mixed Hodge structure on the cohomology of $Y$ \cite{Deligne} and therefore an $E$-series $E(Y; x,y)$. If $Y=[Z/G]$ is a quotient stack of a scheme by a connected algebraic group, then one can show that $E(Y)=E(Z)/E(G)$.
Applying these considerations to the  complex character stack $\fX_\bbC$, our main theorem implies: 
\begin{cor} The complex character stack is polynomial count with counting polynomial $|\!|\fX|\!|_1$. Moreover,  $E(\fX_\bbC; x,y)=|\!|\fX|\!|_1(xy)$. 
\end{cor}

\subsubsection{} This corollary implies that  the virtual Hodge numbers of $\fX_\bbC$ (denoted by $e_{p,q}$ in \cite[Appendix]{HRV}) are balanced; i.e. only  $(p,p)$-type appear. This is in agreement with the (a priori stronger) fact \cite{Shende} that the mixed Hodge structure of $\fX_\bbC$ is Tate; i.e. only $(p,p)$ classes appear.

\subsubsection{} The above discussion implies that corollaries \ref{c:main} and \ref{c:Euler} remain valid in the complex setting. On the other hand, the formulas for the dimension and number of components of $\fX_\bC$ (or, more precisely, its coarse moduli space) can also be understood via the non-abelian Hodge theory. Namely, we have a real analytic isomorphism between the character variety and the moduli space of 
 semistable $G$-Higgs bundles on a compact Riemann surface of genus $g$. 
The formulas for dimension and number of components have been known for a long time in the Higgs bundle setting.

 \end{document}